\documentclass[A4, 11pt]{article}
\usepackage{amscd, amssymb, amsmath, amsthm}
\usepackage{latexsym}
\usepackage{upref}
\usepackage{epsfig}
\usepackage{mathrsfs}
\usepackage{float, floatflt}
\usepackage{indentfirst}
\usepackage{hyperref}
\hypersetup{
    colorlinks=true,
    linkcolor=blue,
    citecolor=blue,
    filecolor=blue,
    urlcolor=blue,
}
\usepackage{graphics,graphicx,color}
\usepackage{cite}
\usepackage{enumitem}
\usepackage{fancyhdr} 
\usepackage[T1]{fontenc}
\usepackage{mathptmx} 

\AtBeginDocument{
  \DeclareSymbolFont{AMSb}{U}{msb}{m}{n}
  \DeclareSymbolFontAlphabet{\mathbb}{AMSb}}  

\usepackage{rotating} 
\usepackage{booktabs} 
\usepackage{pb-diagram,pb-xy}       
\usepackage{tikz}                   
\usetikzlibrary{matrix,arrows}
\usepackage{multirow} 
\usepackage{color, colortbl} 
\definecolor{Gray}{gray}{0.93} 
\definecolor{LightCyan}{rgb}{0.88,1,1}

\usepackage{pgf,tikz,pgfplots}
\usepackage{mathrsfs}
\usetikzlibrary{arrows}

\theoremstyle{plain}
\newtheorem{theorem}{Theorem}[section]
\newtheorem{lemma}[theorem]{Lemma}

\newtheorem{corollary}[theorem]{Corollary}
\theoremstyle{definition}
\newtheorem{definition}[theorem]{Definition}
\newtheorem{example}[theorem]{Example}
\theoremstyle{remark}
\newtheorem*{remark}{Remark}

\numberwithin{equation}{section}

\makeatletter
\def\th@plain{%
  \thm@notefont{}
  \itshape 
}
\def\th@definition{%
  \thm@notefont{}
  \normalfont 
} \makeatother

\setlist{font=\normalfont}

\DeclareMathAlphabet{\cols}{OMS}{cmsy}{m}{n} %

\newcommand{\norm}[1]{\|#1\|}
\newcommand{\set}[1]{\{#1\}}
\newcommand{\cset}[2]{\set{{#1}\colon{#2}}}
\newcommand{\abs}[1]{|#1|}
\newcommand{\gen}[1]{\langle#1\rangle}

\newcommand{\N}{\mathbb{N}}
\newcommand{\Z}{\mathbb{Z}}
\newcommand{\R}{\mathbb{R}}
\newcommand{\aut}[1]{\mathrm{Aut}\,{(#1)}}

\newcommand{\Cay}[1]{\mathrm{Cay}\,{(#1)}}
\newcommand{\caut}[1]{\mathrm{Aut}_c{(#1)}}

\newcommand{\dCay}[1]{\overrightarrow{\mathrm{Cay}}{(#1)}}
\newcommand{\dcCay}[1]{\overrightarrow{\mathrm{Cay}}_{c}{(#1)}}

\newcommand{\diam}[1]{\mathrm{diam}\,{(#1)}}
\newcommand{\Iso}[1]{\mathrm{Iso}\,{(#1)}}
\newcommand{\sym}[1]{\mathrm{Sym}\,{(#1)}}
\newcommand{\paut}[1]{\mathrm{Aut}_p{(#1)}}

\newcommand{\qt}[1]{``#1''}

\begin{document}
\title{\bf Geometry of generated groups with metrics induced by their Cayley color graphs}
\author{Teerapong Suksumran\,\href{https://orcid.org/0000-0002-1239-5586}{\includegraphics[scale=1]{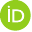}}\\
Research Center in Mathematics and Applied Mathematics\\
Department of Mathematics\\
Faculty of Science, Chiang Mai University\\
Chiang Mai 50200, Thailand\\
{\tt teerapong.suksumran@cmu.ac.th}}
\date{}
\maketitle


\begin{abstract}
Let $G$ be a group and let $S$ be a generating set of $G$. In this article, we introduce a metric $d_C$ on $G$ with respect to $S$, called the cardinal metric. We then compare geometric structures of $(G, d_C)$ and $(G, d_W)$, where $d_W$ denotes the word metric. In particular, we prove that if $S$ is finite, then $(G, d_C)$ and $(G, d_W)$ are not quasi-isometric in the case when $(G, d_W)$ has infinite diameter and they are bi-Lipschitz equivalent otherwise. We also give an alternative description of cardinal metrics by using Cayley color graphs. It turns out that color-permuting and color-preserving automorphisms of Cayley digraphs are isometries with respect to cardinal metrics.
\end{abstract}
\thispagestyle{empty}

\section{Preliminaries}
By a {\it generated} group $(G, S)$ we mean a group $G$ with a specific generating set $S$. In the case when $S$ is finite, we say that $(G, S)$ is a {\it finitely generated} group. Throughout this article, we assume that every generating set of a group does not contain the group identity. One of the most important metrics defined on a \mbox{generated} group is presented below.

\begin{definition}[Word metrics]
Let $(G, S)$ be a generated group. The {\it word metric} with respect to $S$, denoted by $d_W$, is defined on $G$ by
\begin{equation}
d_W(g, h) = \min{\cset{n\in\N}{g^{-1}h = s_1^{\epsilon_1}s_2^{\epsilon_2}\cdots s_n^{\epsilon_n}, s_i\in S, \epsilon_i\in\set{\pm 1}}}
\end{equation}
for all $g, h\in G$ with $g\ne h$ and $d_W(g, g) = 0$ for all $g\in G$.
\end{definition}

Let $(G, S)$ be a generated group. Recall that the {\it Cayley digraph} of $G$ with respect to $S$, denoted by $\dCay{G, S}$, is a directed graph (also called a digraph) such that
\begin{enumerate}[label=(\roman*)]
\item the vertex set is $G$ and
\item the set of arcs is $\cset{(g, gs)}{g\in G, s\in S}$.
\end{enumerate}
The (undirected) {\it Cayley graph} of $G$ with respect to $S$, denoted by $\Cay{G, S}$, is defined as the underlying graph of $\dCay{G, S}$; that is, the vertex sets of $\Cay{G, S}$ and $\dCay{G, S}$ are the same and $\set{g, h}$ is an edge in $\Cay{G, S}$ if and only if $(g, h)$ or $(h, g)$ is an arc in $\dCay{G, S}$. A strong connection between word metrics and Cayley graphs is reflected in the fact that the distance between two arbitrary points $g$ and $h$ in $G$ measured by the word metric coincides with the shortest length of a path joining vertices $g$ and $h$ in $\Cay{G, S}$.

The word metric of a generated group leads to a {\it group-norm} \cite{NBAONVT2010}, which is a function analogous to a norm on a linear space. In fact, if $(G, S)$ is a generated group, then
\begin{equation}
\|g\|_W= \min{\cset{n\in\N}{g = s_1^{\epsilon_1}s_2^{\epsilon_2}\cdots s_n^{\epsilon_n}, s_i\in S, \epsilon_i\in\set{\pm 1}}}
\end{equation}
for all $g\ne e$ defines a group-norm on $G$. For basic knowledge of geometric group theory, we refer the reader to \cite{MR3729310}.

\section{Generated groups endowed with cardinal metrics}
Motivated by word metrics, we introduce a new metric on a group with a \mbox{specific} (finite or infinite) generating set. This metric enriches the group structure and so the group becomes a metric space with a norm-like function. We then examine its geometric structure. In particular, we show that large scale geometry of a certain finitely generated group is different when it is equipped with this new metric instead of the word metric. This enables us to investigate finitely generated groups from another point of view.

Let $G$ be a group and let $S$ be a subset of $G$. Recall that $G = \gen{S}$ if and only if every element of $G$ is of the form $s_1^{\epsilon_1}s_2^{\epsilon_2}\cdots s_n^{\epsilon_n}$, where $s_i\in S$ and $\epsilon_i\in\set{\pm 1}$ for all $i = 1,2,\ldots, n$. Let $(G, S)$ be a generated group. According to the well-ordering principle, we can define a function $\norm{\cdot}$ corresponding to $S$, called the {\it cardinal norm} on $G$, by
\begin{equation}\label{eqn: cardinal norm}
\norm{g} = \min\cset{\abs{A}}{A\subseteq S\textrm{ and }g\in \gen{A}}
\end{equation}
for all $g\in G$. For emphasis and clarity, we sometimes use the notation $\norm{\cdot}_S$. Note that the cardinal norm defined by \eqref{eqn: cardinal norm} does depend on a given generating set $S$. For example, consider the symmetric group $S_3$ and let $S = \set{(1\ 2),  (1\ 2\ 3)}$ and $T = \set{(1\ 2),  (1\ 3), (1\ 2\ 3)}$. Then $S$ and $T$ are generating sets of $S_3$. In this case, $\norm{(1\ 3)}_S = 2$, whereas $\norm{(1\ 3)}_T = 1$.

\begin{theorem}\label{thm: group-norm induced by cardinal metric}
Let $(G, S)$ be a generated group. The cardinal norm induced by $S$ is a group-norm; that is, it satisfies the following properties:
\begin{enumerate}
\item\label{item: positivity} $\norm{g}\geq 0$  for all $g\in G$ and $\norm{g} = 0$ if and only if $g$ is the identity of $G$;
\item\label{item: invariant under taking inverses} $\norm{g^{-1}} = \norm{g}$ for all $g\in G$;
\item\label{item: subadditivity} $\norm{gh}\leq \norm{g}+\norm{h}$ for all $g, h\in G$.
\end{enumerate}
\end{theorem}
\begin{proof}
The proofs of items \ref{item: positivity} and \ref{item: invariant under taking inverses} are straightforward. To prove item \ref{item: subadditivity}, let $g, h\in G$. Then $g\in \gen{T}$ and $h\in \gen{R}$, where $T$ and $R$ are subsets of $S$ such that $\norm{g} = \abs{T}$ and $\norm{h} = \abs{R}$. Note that $T\cup R$ is a finite subset of $S$ and that $gh\in\gen{T\cup R}$. Hence, $\norm{gh}\leq \abs{T\cup R}\leq \abs{T}+\abs{R} = \norm{g}+\norm{h}$.
\end{proof}

It follows from Theorem \ref{thm: group-norm induced by cardinal metric} that the cardinal norm of a generated group $G$ induces a metric given by
\begin{equation}\label{eqn: cardinal metric}
d_C(g, h) = \norm{g^{-1}h},\qquad g,h\in G,
\end{equation}
called the {\it cardinal metric}, and so $G$ becomes a metric space. This gives another way to define a geometric structure on a (finitely) generated group. One of the most important geometric structures defined on a finitely generated group is the word metric, of course. By  definition, $d_C(g, h)$ equals the smallest cardinality of a subset $A$ of $S$ such that $g^{-1}h\in\gen{A}$, which justifies the use of the term \qt{cardinal}. It will be apparent that the structure of a generated group depends on its diameter with respect to the cardinal metric as well as the word metric. 

\begin{remark}
Throughout the remainder of this article, the word and cardinal metrics mentioned in the same place are induced by the same given generating set unless stated otherwise.
\end{remark}

\begin{lemma}\label{lem: upper bound for cardinal metric}
If $(G, S)$ is a finitely generated group, then 
\begin{equation}
d_C(g, h)\leq \abs{S}
\end{equation}
for all $g, h\in G$.
\end{lemma}
\begin{proof}
The lemma follows from the fact that $a\in\gen{S}$ for all $a\in G$.
\end{proof}

The following theorem shows that every finitely generated group has finite \mbox{diameter} with respect to the cardinal metric. Furthermore, there is an example of a generated group of infinite diameter.

\begin{theorem}\label{thm: sufficient condition to be finite diameter}
Let $(G, S)$ be a generated group. If $S$ contains a finite subset that generates $G$, then $(G, d_C)$ is of finite diameter.
\end{theorem}
\begin{proof}
Suppose that $T$ is a finite subset of $S$ such that $\gen{T} = G$. As in Lemma \ref{lem: upper bound for cardinal metric}, $d_C(g, h)\leq \abs{T}$ for all $g, h\in G$. Hence, $\sup{\cset{d_C(g,h)}{g, h\in G}} < \infty$ and so $\diam{G, d_C}$ is finite.
\end{proof}

The converse to Theorem \ref{thm: sufficient condition to be finite diameter} does not hold. For example, $(\R, \R)$ is a generated group and $(\R, d_C)$ is of finite diameter since $d_C(x, y)\leq 1$ for all $x, y\in \R$. However, $\R$ dose not contain a finite subset that generates $\R$.

\begin{example}
Let $F$ be a free abelian group with infinite countable basis $B = \cset{b_i}{i\in\N}$ (For example, $F$ can be chosen as the group of all functions from $\N$ to $\Z$ with finitely many nonzero values under pointwise addition). For each $n\in\N$, define $s_n = b_1+b_2+\cdots+b_n$. Then $s_n\in\gen{b_1, b_2,\ldots, b_n}$ and so $d_C(0, s_n)\leq n$. 
Suppose to the contrary that $d_C(0, s_n) = m < n$. Then there are distinct elements $c_1, c_2, \ldots, c_m\in B$ such that $s_n\in\gen{c_1, c_2,\ldots, c_m}$. Hence,
$$
b_1+b_2+\cdots+b_n = k_1c_1 + k_2c_2+\cdots +k_mc_m
$$
for some $k_1, k_2,\ldots, k_m\in\Z$. Since $m<n$, the previous equation is a contradiction for $b_1, b_2,\ldots, b_n, c_1, c_2, \ldots, c_m$ are basis elements. This proves that $d_C(0, s_n) = n$. It follows that $\sup{\cset{d_C(x, y)}{x, y\in F}} = \infty$ and so $(F, d_C)$ is of infinite diameter.
\end{example}

\subsection{Geometric structures}
Let $(G, S)$ be a generated group. It is not difficult to check that the following are classes  of (surjective) isometries of $G$ with respect to the cardinal metric:
\begin{itemize}
\item the left multiplication maps $L_g\colon h\mapsto gh$;
\item the automorphisms $\tau$ of $G$ with the property that $\norm{\tau(g)} = \norm{g}$ for all $g\in G$;
\item the automorphisms $\tau$ of $G$ with the property that $\tau(S) = S$.
\end{itemize}
An immediate consequence of the previous result is that the space $(G, d_C)$ is homo-geneous; that is, if $x$ and $y$ are arbitrary points of $G$, then there is an isometry $T$ of $(G, d_C)$ for which $T(x) = y$. In fact, $T = L_{yx^{-1}}$ is the desired isometry.

As mentioned previously, the cardinal metric depends on its generating set. However, in the case of finitely generated groups, the cardinal metrics are unique up to bi-Lipschitz equivalence, as shown in the following theorem.

\begin{theorem}
Let $G$ be a group with finite generating sets $S$ and $T$ and let $d_S$ and $d_T$ be the cardinal metrics induced by $S$ and $T$, respectively. Every injective self-map of $G$ is bi-Lipschitz. In particular, every permutation of $G$ is a bi-Lipschitz equivalence and so $(G, d_S)$ and $(G, d_T)$ are bi-Lipschitz equivalent.
\end{theorem}
\begin{proof}
We may assume without loss of generality that $\abs{S}\leq \abs{T}$. Suppose that $f\colon G\to G$ is an injective map. Let $g, h\in G$ and let $g\ne h$. Then $f(g)\ne f(h)$ and so $d_S(f(g), f(h))>0$. By the defining property of $d_S$, $d_S(f(g), f(h))\geq 1$. By Lemma \ref{lem: upper bound for cardinal metric}, $d_T(g, h) < \abs{T}+1$. Hence, $\dfrac{1}{\abs{T}+1}d_T(g, h) < 1\leq d_S(f(g), f(h))$. By the same lemma, $d_S(f(g), f(h)) \leq \abs{S}\leq\abs{T} < (\abs{T}+1)d_T(g, h)$. This proves that
$$
\dfrac{1}{\abs{T}+1} d_T(g, h) \leq d_S(f(g), f(h)) \leq (\abs{T}+1)d_T(g, h)
$$
and so $f$ is bi-Lipschitz. The remaining part of the theorem is immediate.
\end{proof}

In some instances large scale geometry of a generated group is variant when its word metric is replaced by the cardinal metric, as we will see shortly. The next theorem gives a comparison between cardinal and word metrics.

\begin{theorem}\label{thm: comparison between word and cardinal metric}
Let $(G, S)$ be a generated group. Then
\begin{equation}\label{eqn: comparison of dW and dC}
d_C(g, h)\leq d_W(g, h)
\end{equation}
for all $g, h\in G$. Further, there is an example of a group such that the equality in \eqref{eqn: comparison of dW and dC} does not hold.
\end{theorem}
\begin{proof}
Let $g, h\in G$. If $g = h$, then $d_C(g, h) = 0 = d_W(g,h)$. Suppose that $g\ne h$ and that $d_W(g, h) = m$. Then there are elements $s_1, s_2, \ldots, s_m$ in $S$ such that $g^{-1}h = s_1^{\epsilon_1}s_2^{\epsilon_2}\cdots s_m^{\epsilon_m}$, where $\epsilon_i\in\set{\pm 1}$ for all $i = 1,2,\ldots, m$. Thus, $g^{-1}h\in \gen{s_1, s_2,\ldots, s_m}$ and so $\norm{g^{-1}h}\leq m$. Hence, $d_C(g, h)\leq  d_W(g, h)$.

For the remaining part of the theorem, consider the additive group $\Z$. Since $\Z = \gen{1}$, it follows that $\norm{k} = 1$ for all nonzero $k\in\Z$. Hence, $d_C(m, n) \leq 1$ for all $m, n\in\Z$. If $-m + n\geq 2$, then $d_W(m, n)\geq 2$. This proves that $d_C(m, n)<d_W(m, n)$ whenever $-m+n\geq 2$.
\end{proof}

\begin{theorem}\label{thm: dW and dC, infinite diameter}
Let $(G, S)$ be a finitely generated group. If $(G, d_W)$ is of infinite diameter, then $(G, d_W)$ and $(G, d_C)$ are not quasi-isometric.
\end{theorem}
\begin{proof}
We show that there is no quasi-isometric embedding from $(G, d_W)$ to $(G, d_C)$. Let $T$ be a self-map of $G$. Let $K$ and $c$ be arbitrary positive constants. Since $\diam{G, d_W} =\sup{\cset{d_W(x, y)}{x, y\in G}}= \infty$ and $K(\abs{S}+c)$ is a constant, there must be points $g$ and $h$ in $G$ such that $d_W(g, h) > K(\abs{S}+c)$. It follows that $\dfrac{1}{K}d_W(g, h) - c > \abs{S}$. By Lemma \ref{lem: upper bound for cardinal metric}, $d_C(T(g), T(h)) \leq \abs{S}$. Hence, $$\dfrac{1}{K}d_W(g, h) - c > d_C(T(g), T(h)).$$ This proves that $T$ cannot define a quasi-isometric embedding and so $(G, d_W)$ and $(G, d_C)$ are not quasi-isometric. 
\end{proof}

\begin{theorem}\label{thm: bilipschitz equivanet, depending on diameter}
Let $G$ be a finitely generated group.
\begin{enumerate}
\item\label{item: infinite diameter} If $(G, d_W)$ is of infinite diameter, then $(G, d_W)$ and $(G, d_C)$ are not bi-Lipschitz equivalent.
\item\label{item: finite diameter} If $(G, d_W)$ is of finite diameter, then $(G, d_W)$ and $(G, d_C)$ are bi-Lipschitz equivalent. Therefore, they are quasi-isometric.
\end{enumerate}
\end{theorem}
\begin{proof}\hfill

\eqref{item: infinite diameter} If $(G, d_W)$ is of infinite diameter, then by Theorem \ref{thm: dW and dC, infinite diameter}, $(G, d_W)$ and $(G, d_C)$ are not quasi-isometric and so they are not bi-Lipschitz equivalent. 

\eqref{item: finite diameter} Suppose that $(G, d_W)$ is of finite diameter and let $T$ be an injective self-map of $G$. We claim that $T$ is a bi-Lipschitz embedding. Set $K = \diam{G, d_W}$. Using Theorem \ref{thm: comparison between word and cardinal metric}, we obtain
$$
\dfrac{1}{K}d_W(g, h)\leq d_C(T(g), T(h)) \leq Kd_W(g, h)
$$
for all $g, h\in G$. Hence, $T$ is bi-Lipschitz. This implies that every permutation of $G$ is a bi-Lipschitz equivalence between $(G, d_W)$ and $(G, d_C)$. 
\end{proof}

By Theorem \ref{thm: bilipschitz equivanet, depending on diameter} \eqref{item: finite diameter}, if $(G, d_W)$ has finite diameter, then $(G, d_W)$ and $(G, d_C)$ are bi-Lipschitz equivalent. Therefore, the next obvious question is whether they are isometric. It turns out that they need not be isometric, in general. The following two examples support our claim.

\begin{example}\label{exm: isometric (G, dW) and (G, dC)}
Let $G$ be a finite group. Let $d_W$  and $d_C$ be the word and cardinal metrics with respect to $G$ itself. Note that $(G, d_W)$ is of finite diameter. In fact, if $g, h\in G$, then $d_W(g, h) = 1$ since $\Cay{G, G}$ is a complete graph and so $\set{g,h}$ is an edge in $\Cay{G, G}$. This implies that the diameter of $(G,d_W)$ equals $1$. The identity map on $G$ is easily seen to be an isometry between $(G, d_W)$ and $(G, d_C)$.
\end{example}

\begin{example}\label{exm: not isometric (G, dW) and (G, dC)}
Let $G$ be a cyclic group of finite order $n \geq 4$ with a generator $a$. Let $d_W$  and $d_C$ be the word and cardinal metrics induced by $\set{a}$ and let $T$ be a bijection from $G$ to itself. We claim that $T$ cannot define an isometry between $(G, d_W)$ and $(G, d_C)$. Since $T$ is surjective, there are elements $g$ and $h$ of $G$ such that $T(g) = e$ and $T(h) = a^2$. Note that $d_W(T(g), T(h)) = d_W(e, a^2) = 2$ since $e^{-1}a^2 = a^2$ is a word of length $2$ ($a^2\ne e$, $a^2\ne a$, and $a^2\ne a^{-1}$). Moreover, $d_C(g, h) \leq 1$ since $g^{-1}h\in G = \gen{a}$. Hence, $d_W(T(g), T(h))\ne d_C(g, h)$ and so $T$ is not an isometry. 
\end{example}

We close this section with a description of isometries between finitely \mbox{generated} groups equipped with word and cardinal metrics.

\begin{theorem}
Let $(G, S)$ be a finitely generated group. Every isometry from $(G, d_C)$ to $(G, d_W)$ is of the form 
$L_a\circ\tilde{T}$, where $a\in G$ and $\tilde{T}\colon (G, d_C)\to (G, d_W)$  is an isometry that preserves the group identity and $\tilde{{T}}(S)\subseteq S\cup S^{-1}$. Furthermore, $\tilde{T}$ is a nonexpansive mapping on $(G, d_W)$.
\end{theorem}
\begin{proof}
Denote by $e$ the identity of $G$ and let $T\colon (G, d_C)\to (G, d_W)$ be an isometry. Define $\tilde{T} = L_{T(e)^{-1}}\circ T$. Then $\tilde{T}(e) = T(e)^{-1}T(e) = e$. Let $g, h\in G$. Since $L_{T(e)^{-1}}$ is an isometry of $(G, d_W)$, it follows that
$$
d_W(\tilde{T}(g), \tilde{T}(h)) = d_W(T(g), T(h)) = d_C(g, h).
$$
Hence, $\tilde{T}$ is an isometry from $(G, d_C)$ to $(G, d_W)$. Let $s\in S$. Then
$$
1 = d_C(e, s) = d_W(\tilde{T}(e), \tilde{T}(s)) = d_W(e, \tilde{T}(s)).
$$
This implies that $\tilde{T}(s)$ can be expressed as a word of length $1$; that is, $\tilde{T}(s) = t^{\epsilon}$ for some $t\in S$ and $\epsilon\in\set{\pm 1}$. This proves that $\tilde{T}(S)\subseteq S\cup S^{-1}$. 

Let $g,h\in G$. By Theorem \ref{thm: comparison between word and cardinal metric},
$$
d_W(\tilde{T}(g), \tilde{T}(h)) = d_C(g, h) \leq d_W(g, h)
$$
and so $\tilde{T}$ is nonexpansive.
\end{proof}

\subsection{Topological structures}
Let $G$ be a generated group. It is clear that the distance between two arbitrary points in $G$ measured by the cardinal metric is a nonnegative integer and so the cardinal metric induces the {\it discrete} topology on $G$ (the same is true for the word metric). Actually, the open ball centered at $x$ of radius $1/2$ is the singleton set $\set{x}$. This implies that if $G$ is finite, then $(G, d_C)$ is compact and hence is complete and totally bounded. In contrast, if $G$ is infinite, then $(G, d_C)$ is neither compact nor totally bounded. Nevertheless, it is complete since any Cauchy sequence in $(G, d_C)$ must become constant at some point. It is well known that any finitely generated group is countable. Therefore, if $G$ is a finitely generated group, then $(G, d_C)$ is separable since $G$ is a countable dense subset of itself. In this case, $(G, d_C)$ forms a {\it Polish} metric space \cite{GBPTC1991} (and even a Polish group).

\section{Isometries of cardinal metrics}
In this section, we give an alternative description of cardinal metrics by using Cayley color graphs. This leads to a remarkable connection between cardinal metrics and color-permuting automorphisms of Cayley graphs of generated groups. More precisely, cardinal metrics are invariant under color-permuting automorphisms (and hence also color-preserving automorphisms).

Let $(G, S)$ be a generated group. To elements of $S$, we can associate distinct colors, labeled by their names. The (right) {\it Cayley color digraph} of $G$ with \mbox{respect} to $S$, denoted by $\dcCay{G, S}$, is a digraph with $G$ as the vertex set and for all $g, h\in G$, there is an arc  from $g$ to $h$ if and only if $h = gc$ for some $c\in S$. In this case, we say that $(g, h)$ is an arc with color $c$. Recall that an (undiretced) {\it path} from $g$ to $h$ in $\dcCay{G, S}$ is an alternating sequence of vertices and arcs, $g = g_0, e_1, g_1, e_2, g_2,\ldots, g_{n-1}, e_n, g_n = h$, such that $e_i\in\set{(g_{i-1}, g_i), (g_i, g_{i-1})}$ for all $i = 1,2,\ldots, n$.

\begin{theorem}[An alternative description of cardinal metrics]\label{thm characterization of cardinal metric}
Let $(G, S)$ be a generated group and let $g$ and $h$ be distinct elements of $G$. Then $d_C(g, h)$ equals the minimum number of colors associated with a path connecting $g$ and $h$ in $\dcCay{G, S}$.
\end{theorem}
\begin{proof}
Let $n$ be the minimum number of colors associated with a path connecting $g$ and $h$ in $\dcCay{G, S}$. Then there is a path $g = x_0, x_1,\ldots, x_m = h$ in $\dcCay{G, S}$ with $c_1, c_2,\ldots, c_n$ as its colors. It follows that $g^{-1}h = a_1^{\epsilon_1}a_2^{\epsilon_2}\cdots a_m^{\epsilon_m}$, where $a_i\in\set{c_1, c_2, \ldots, c_n}$ and $\epsilon_i\in\set{\pm 1}$ for all $i = 1,2,\ldots, m$. This implies that $g^{-1}h\in\gen{c_1, c_2,\ldots, c_n}$. Hence, $d_C(g, h)\leq n$.

By definition, there is a subset $T$ of $S$ with $\abs{T} = d_C(g,h)$ such that $g^{-1}h\in\gen{T}$. It follows that $g^{-1}h = b_1^{\epsilon_1}b_2^{\epsilon_2}\cdots b_k^{\epsilon_k}$, where $b_i\in T$ and $\epsilon_i\in\set{\pm 1}$ for all $i = 1,2,\ldots, k$. Further, we may assume that $b_1^{\epsilon_1}b_2^{\epsilon_2}\cdots b_k^{\epsilon_k}$ does not contain a subword equal to the group identity. Define $y_0 = g$ and $y_i = y_{i-1}b_i^{\epsilon_i}$. Then $g = y_0, y_1, y_2,\ldots, y_k = h$ is a path connecting $g$ and $h$ in $\dcCay{G, S}$ and so the \mbox{number} of colors associated with this path, say $\ell$, does not exceed $\abs{T}$. It follows from the minimality of $n$ that $n\leq \ell\leq \abs{T} = d_C(g, h)$. Thus, $d_C(g, h)=n$.
\end{proof}

Let $(G, S)$ be a generated group. Denote by $\aut{\dcCay{G, S}}$ the group of graph automorphisms of $\dcCay{G, S}$. 

\begin{definition}[Color-permuting automorphisms, \cite{MR3546658}]\label{def: color-permuting automorphisms}
Let $(G, S)$ be a generated group. A map $\alpha$ in $\aut{\dcCay{G, S}}$ is called a {\it color-permuting} automorphism of $\dcCay{G, S}$ if there exists a permutation $\sigma\in\sym{S}$ such that $(g, h)$ has color $c$ if and only if $(\alpha(g), \alpha(h))$ has color $\sigma(c)$ for all $g, h$ in $G$. 
\end{definition}
A natural characterization of color-permuting automorphisms of $\dcCay{G, S}$ is presented in the following theorem.

\begin{theorem}[p. 66, \cite{MR1206550}]\label{thm: characterization of color-permuting automorphism}
Let $(G, S)$ be a generated group and let $\alpha$ be an automorphism of $\dcCay{G, S}$. Then $\alpha$ is a color-permuting automorphism of $\dcCay{G, S}$ if and only if there exists a permutation $\sigma\in\sym{S}$ such that $\alpha(gc) = \alpha(g)\sigma(c)$
for all $g\in G$ and $c\in S$.
\end{theorem}

Clearly, the color-permuting automorphisms of $\dcCay{G, S}$ corresponding to the identity permutation of $S$ are precisely the {\it color-preserving} automorphisms of $\dcCay{G, S}$. It is a standard result in graph theory that the group of all color-preserving automorphisms of $\dcCay{G, S}$, denoted by $\caut{\dcCay{G, S}}$, is iso-morphic to $G$ \cite{RFHVG1939}. In fact,
\begin{equation}
\caut{\dcCay{G, S}} = \cset{L_a}{a\in G},\qquad L_a\colon g\mapsto ag,
\end{equation}
and $\cset{L_a}{a\in G}$ is isomorphic to $G$ by the famous Cayley theorem in abstract algebra; see, for instance, Theorem 7.12 of \cite{MR2014620}. 

From the characterization of a cardinal metric described in Theorem \ref{thm characterization of cardinal metric}, we immediately obtain a class of isometries of $(G, d_C)$:

\begin{theorem}
If $(G, S)$ is a generated group, then the color-preserving automorphisms of $\dcCay{G, S}$  are isometries of $G$ with respect to the cardinal metric.
\end{theorem}

Denote by $\paut{\dcCay{G, S}}$ the group of {\it color-permuting} automorphisms of $\dcCay{G, S}$. The well-known characterization of $\paut{\dcCay{G, S}}$ is given by
\begin{equation}
\paut{\dcCay{G, S}} = \cset{L_a\circ \tau}{a\in G, \tau\in\aut{G, S}},
\end{equation}
where $\aut{G, S}$ denotes the group of automorphisms $\tau$ of $G$ such that $\tau(S) = S$; see, for instance, \cite[Lemma 2.1]{MR1206550}. In view of Theorem \ref{thm characterization of cardinal metric}, we obtain another class of isometries of $(G, d_C)$.

\begin{theorem}
If $(G, S)$ is a generated group, then the color-permuting automorphisms of $\dcCay{G, S}$  are isometries of $G$ with respect to the cardinal metric.
\end{theorem}

\begin{corollary}\label{cor: class of isometry La and automorphism}
If $(G, S)$ is a generated group, then 
\begin{equation}
\cset{L_a\circ \tau}{a\in G, \tau\in\aut{G, S}}\subseteq\Iso{G, d_C}.
\end{equation}
\end{corollary}

In general, the inclusion in Corollary \ref{cor: class of isometry La and automorphism} is proper. For instance, if $G = \Z$, the additive infinite cyclic group, then there exists an isometry of $(\Z, d_C)$ that is not an automorphism of $\Z$. In fact, define a map $T$ by $T(2) = 3, T(3) = 2$, and $T(x) = x$ for all $x\in\Z\setminus\set{2, 3}$. It is clear that $T$ is a bijection from $\Z$ to itself. If $x = y$, then $T(x) = T(y)$ and so
$$
d_C(T(x), T(y)) = 0 = d_C(x, y).
$$
If $x\ne y$, then $T(x)\ne T(y)$ and so 
$$
d_C(T(x), T(y)) = 1 = d_C(x, y).
$$
This proves that $T$ is an isometry of $(\Z, d_C)$. However, $T$ does not define an automorphism of $\Z$; for example, $T(4) = 4$, whereas $T(2)+T(2) = 6$.

\vspace{0.3cm}
\noindent{\bf Acknowledgements.} This work was financially supported by the Research Center in Mathematics and Applied Mathematics, Chiang Mai University. The author would like to thank the referee for comments that improve Examples \ref{exm: isometric (G, dW) and (G, dC)} and \ref{exm: not isometric (G, dW) and (G, dC)}.

\bibliographystyle{amsplain}\addcontentsline{toc}{section}{References}
\bibliography{References}
\end{document}